\documentclass{arXiv}

\usepackage{amssymb}
\usepackage{graphicx}
\usepackage[cmtip,all]{xy}

\newtheorem{theorem}{Theorem}[section]
\newtheorem{lemma}[theorem]{Lemma}
\newtheorem{proposition}[theorem]{Proposition}

\theoremstyle{definition}

\theoremstyle{remark}

\numberwithin{equation}{section}

\begin{document}

\title[Semilinear Neumann Problem]
{Rigidity for a semilinear Neumann problem with exponential nonlinearity in the large diffusion limit}

\author[J. Seo]{Juneyoung Seo}
\address{Department of Mathematics and Institute of Mathematical Science,
Pusan National University, Busan 46241, Republic of Korea}
\email{juneys@pusan.ac.kr}
\thanks{The author was supported by the National Research Foundation of Korea grant funded by the Korea government
(MSIT) (RS-2022-NR072398).}


\date{\today}

\begin{abstract}
We consider a semilinear Neumann problem with exponential nonlinearity in a smooth bounded domain $\Omega \subset \mathbb{R}^2$. We prove that there exists a threshold $\bar{\varepsilon}>0$ such that for all $\varepsilon>\bar{\varepsilon}$, any classical solution must be constant. This result provides a positive answer to a conjecture recently posed by Calanchi, Ciraolo, and Messina (2026). Our proof relies on a combination of $L^1$-estimates, a Jensen-type argument via the Neumann Green's function to obtain uniform exponential integrability, and elliptic regularity.
\end{abstract}

\maketitle

\section{Introduction}
This paper is concerned with the semilinear Neumann problem
\begin{equation} \label{eq:main}
\begin{cases}
-\varepsilon \Delta u = e^u - 1 - au & \text{in } \Omega, \\
\frac{\partial u}{\partial \nu} = 0 & \text{on } \partial\Omega,
\end{cases}
\end{equation}
on a smooth bounded domain $\Omega\subset\mathbb{R}^2$, where $\nu$ denotes the outward unit normal, and $a > 1$ and $\varepsilon > 0$ are constants. Recently, Calanchi, Ciraolo, and Messina~\cite{CCM} showed that for sufficiently small $\varepsilon$ the problem admits nonconstant mountain-pass solutions. Furthermore, they conjectured that all classical solutions are constant for $\varepsilon$ sufficiently large. In this paper, we give a positive answer to this conjecture.

The stationary problem \eqref{eq:main} arises naturally as the steady-state equation of reaction-diffusion dynamics and is closely related to spatial pattern formation. Solutions are shaped by the competition between diffusion and nonlinear reaction. The parameter $\varepsilon>0$ denotes the diffusion rate. 

The classical rigidity result of Casten and Holland~\cite{CH}, and independently Matano~\cite{Ma}, asserts that stable stationary solutions of semilinear Neumann problems on bounded convex domains are constant. This result highlights the influence of the domain geometry on the existence of nonconstant solutions.

The magnitude of the diffusion coefficient itself also plays a decisive role in the rigidity of solutions. In their seminal paper, Lin, Ni, and Takagi~\cite{LNT} proved that for power-type nonlinearities with subcritical Sobolev exponents, if the diffusion coefficient is sufficiently large, the corresponding semilinear Neumann problem admits only the constant solution. Physically, this means that in a strongly diffusion-dominated regime, local concentrated patterns cannot emerge, and the system inevitably settles into a spatially uniform steady state.

Conversely, when the diffusion rate is small, the system exhibits highly localized nonconstant patterns. Shifting the focus to this singular perturbation regime ($\varepsilon\to0$), Lin, Ni, and Takagi~\cite{LNT} studied the following problem:
\begin{equation} \label{eq:power}
\begin{cases}
-\varepsilon \Delta u + u = u^p & \text{in } \Omega, \\
u>0 & \text{in } \Omega, \\
\frac{\partial u}{\partial \nu} = 0 & \text{on } \partial\Omega,
\end{cases}
\end{equation}
They showed that for such power-type nonlinearities, as $\varepsilon\to0,$ \eqref{eq:power} admits a nonconstant mountain-pass type solution. Shortly after, Ni and Takagi~\cite{NT1, NT2} proved that this least-energy solution possesses a single spike, and its peak is necessarily located on the boundary $\partial\Omega$. Furthermore, through a rigorous asymptotic expansion of the energy, they established that as $\varepsilon \to 0$, this peak concentrates precisely at the maximum point of the mean curvature of the boundary, establishing a deep link between the nonlinearity and the domain's boundary geometry.

In two dimensions, exponential nonlinearities are governed by the Moser--Trudinger inequality, and the analysis is considerably more delicate than in the polynomial-growth case. Such nonlinearities arise in several models, including chemotaxis, mean field equations for Euler flows, self-dual Chern–Simons–Higgs models, and prescribed Gaussian curvature problems; see \cite{CL, DJLW, KS} and the references therein. 

The main contribution of this paper is to establish uniform a priori $L^{\infty}$-bounds independent of the diffusion parameter in the presence of exponential nonlinearity. 
This allows us to combine nonlinear estimates with spectral arguments to obtain complete rigidity in the large diffusion regime. To the best of our knowledge, this is the first result establishing large diffusion rigidity for semilinear Neumann problems with exponential nonlinearities in two dimensions.

Our main result is stated as follows:

\begin{theorem}\label{thm:main}
Let $a > 1$. There exists a constant $\bar{\varepsilon} > 0$, depending only on $a$ and $\Omega$, such that for every $\varepsilon > \bar{\varepsilon}$, any classical solution $u$ to \eqref{eq:main} must be a constant. Consequently, the only solutions are $u \equiv 0$ and $u \equiv \xi_a$, where $\xi_a > 0$ is the unique positive root of $e^t - 1 - at = 0$.
\end{theorem}

Let $f(u) := e^u - 1 - au$. Integrating \eqref{eq:main} over $\Omega$ yields the identity:
\begin{equation} \label{eq:zero_average}
\int_\Omega f(u) \, dx = 0.
\end{equation}
This identity plays a key role in establishing uniform a priori bounds independent of $\varepsilon$. We also decompose the solution as $u = \bar{u} + v$, where $\bar{u}=\frac{1}{|\Omega|}\int_{\Omega}u\ dx$ is the spatial average of $u$.

\section{Uniform A Priori Estimates}
In this section, we establish uniform $L^\infty$-bounds for the solutions of \eqref{eq:main} as $\varepsilon \to \infty$. Due to the critical exponential growth in dimension two, standard bootstrap arguments are not directly applicable. We overcome this by exploiting the Green's function and Jensen's inequality.

\begin{lemma}\label{lem:L1_bound}
There exists a constant $C_1> 0$, depends only on $a>1$ and $\Omega$, such that
\begin{equation*}
\|f(u)\|_{L^1(\Omega)} \le C_1.
\end{equation*}
\end{lemma}
\begin{proof}
Integrating the equation \eqref{eq:main} over $\Omega$ and using the homogeneous Neumann boundary condition, we obtain by the divergence theorem:
\begin{equation} \label{eq:integral_zero}
\int_\Omega f(u) \, dx = -\varepsilon \int_\Omega \Delta u \, dx = -\varepsilon \int_{\partial\Omega} \frac{\partial u}{\partial \nu} \, dS = 0.
\end{equation}
Note that the function $f(t) = e^t - 1 - at$ attains its global minimum at $t = \log a$. Let $-C_0$ denote this minimum value, that is,
\begin{equation*}
f(t) \ge f(\log a) = a - 1 - a \log a := -C_0 \quad \text{for all } t \in \mathbb{R}.
\end{equation*}

Let us decompose $f(u)$ into its positive and negative parts, $f(u) = f^+(u) - f^-(u)$, where $f^{\pm}(u)\ge0$.
The global lower bound on $f$ implies that $0 \le f^-(u) \le C_0$ pointwise in $\Omega$.
From \eqref{eq:integral_zero}, we have
\begin{equation*}
\int_\Omega f^+(u) \, dx = \int_\Omega f^-(u) \, dx \le \int_\Omega C_0 \, dx = C_0 |\Omega|.
\end{equation*}

Consequently, the $L^1$-norm of $f(u)$ can be estimated as follows:
\begin{equation*}
\|f(u)\|_{L^1(\Omega)} = \int_\Omega f^+(u) \, dx + \int_\Omega f^-(u) \, dx \le 2 C_0 |\Omega|.
\end{equation*}
Setting $C_1 = 2 C_0 |\Omega|$, we see that $C_1$ is independent of $\varepsilon$ and the solution $u$. This proves the lemma.
\end{proof}

Next, we represent the fluctuation $v = u - \bar{u}$ using the Green's function $G(x,y)$:
\begin{equation*}
    \begin{cases}
-\Delta_x G(x,y) = \delta_y-\frac{1}{|\Omega|} & \text{in } \Omega, \\
\frac{\partial G}{\partial \nu} = 0 & \text{on } \partial\Omega,
\end{cases}
\end{equation*}
with normalization $\int_{\Omega}G(x,y)\ dx = 0$ for each $y\in\Omega$.
Using the logarithmic singularity of $G(x,y)$ in dimension two and the vanishing $L^1$ mass of the right-hand side $\frac{1}{\varepsilon}f(u)$, and following the Brezis--Merle approach \cite{BM}, we obtain the following exponential integrability estimate.

\begin{lemma}\label{lem:jensen}
For any $q > 2$, there exist constants $C_q > 0$ and $\varepsilon_0 > 0$ such that for all $\varepsilon > \varepsilon_0$,
\begin{equation*}
\int_\Omega e^{q|u(x) - \bar{u}|} \, dx \le C_q.
\end{equation*}
\end{lemma}
\begin{proof}
Let $v = u - \bar{u}$. Then $v$ satisfies the following Neumann problem:
\begin{equation*}
\begin{cases}
-\Delta v = f_\varepsilon(u) := \frac{1}{\varepsilon} f(u) & \text{in } \Omega, \\
\frac{\partial v}{\partial \nu} = 0 & \text{on } \partial\Omega.
\end{cases}
\end{equation*}
Let $G(x,y)$ be the Green's function for the Neumann Laplacian on $\Omega$.
Since $\Omega \subset \mathbb{R}^2$ is a bounded smooth domain, the fundamental
solution of the Laplacian has a logarithmic singularity, and the Green's function
admits the decomposition $G(x,y) = -\frac{1}{2\pi}\log\frac{1}{|x-y|} + H(x,y)$,
where $H$ is bounded on $\overline{\Omega} \times \overline{\Omega}$.
In particular, there exist constants $D = \mathrm{diam}(\Omega)$ and $K > 0$
depending only on $\Omega$ such that
\begin{equation*}
|G(x,y)| \le \frac{1}{\pi} \log \left( \frac{D}{|x-y|} \right) + K
\quad \text{for all } x, y \in \Omega,\ x \neq y.
\end{equation*}
Using the representation formula $v(x) = \int_\Omega G(x,y) f_\varepsilon(u(y)) \, dy$ and following the classical strategy introduced in \cite[Theorem 1]{BM}, we can estimate the absolute value of the fluctuation $v(x)$ as follows:
\begin{equation*}
|v(x)| \le \int_\Omega \left( \frac{1}{\pi} \log \left( \frac{D}{|x-y|} \right) + K \right) |f_\varepsilon(u(y))| \, dy.
\end{equation*}
If $\|f_\varepsilon(u)\|_{L^1(\Omega)} = 0$, then $v \equiv 0$ and the lemma holds trivially. Thus, we may assume $\|f_\varepsilon(u)\|_{L^1(\Omega)} > 0$. Dividing both sides by $\|f_\varepsilon(u)\|_{L^1(\Omega)}$ and multiplying by $\pi$, we obtain
\begin{equation*}
\frac{\pi}{\|f_\varepsilon(u)\|_{L^1(\Omega)}} |v(x)| \le \int_\Omega \log \left( \frac{D}{|x-y|} \right) \frac{|f_\varepsilon(u(y))|}{\|f_\varepsilon(u)\|_{L^1(\Omega)}} \, dy + \pi K.
\end{equation*}
Since $\frac{|f_\varepsilon(u(y))|}{\|f_\varepsilon(u)\|_{L^1(\Omega)}} \, dy$ defines a probability measure on $\Omega$, we can apply Jensen's inequality to the convex function $t \mapsto e^t$:
\begin{align*}
\exp\left( \frac{\pi}{\|f_\varepsilon(u)\|_{L^1(\Omega)}} |v(x)| \right) 
&\le e^{\pi K} \exp \left( \int_\Omega \log \left( \frac{D}{|x-y|} \right) \frac{|f_\varepsilon(u(y))|}{\|f_\varepsilon(u)\|_{L^1(\Omega)}} \, dy \right) \\
&\le e^{\pi K} \int_\Omega \exp \left( \log \left( \frac{D}{|x-y|} \right) \right) \frac{|f_\varepsilon(u(y))|}{\|f_\varepsilon(u)\|_{L^1(\Omega)}} \, dy \\
&= e^{\pi K} \int_\Omega \frac{D}{|x-y|} \frac{|f_\varepsilon(u(y))|}{\|f_\varepsilon(u)\|_{L^1(\Omega)}} \, dy.
\end{align*}
Integrating both sides with respect to $x$ over $\Omega$ and applying Fubini's theorem yields
\begin{align*}
\int_\Omega \exp\left( \frac{\pi}{\|f_\varepsilon(u)\|_{L^1(\Omega)}} |v(x)| \right) \, dx 
&\le e^{\pi K} \int_\Omega \left( \int_\Omega \frac{D}{|x-y|} \, dx \right) \frac{|f_\varepsilon(u(y))|}{\|f_\varepsilon(u)\|_{L^1(\Omega)}} \, dy \\
&\le e^{\pi K} C_2 \int_\Omega \frac{|f_\varepsilon(u(y))|}{\|f_\varepsilon(u)\|_{L^1(\Omega)}} \, dy = C_2 e^{\pi K},
\end{align*}
where $C_2 := \sup_{y \in \Omega} \int_\Omega \frac{D}{|x-y|} \, dx \le 2\pi D^2$ by a direct calculation in polar coordinates.
Now, by Lemma \ref{lem:L1_bound}, we have the uniform bound $\|f(u)\|_{L^1(\Omega)} \le C_1$, which implies
\begin{equation*}
\|f_\varepsilon(u)\|_{L^1(\Omega)} = \frac{1}{\varepsilon} \|f(u)\|_{L^1(\Omega)} \le \frac{C_1}{\varepsilon}.
\end{equation*}
Consequently, we obtain the lower bound for the exponent coefficient:
\begin{equation*}
\frac{\pi}{\|f_\varepsilon(u)\|_{L^1(\Omega)}} \ge \frac{\pi \varepsilon}{C_1}.
\end{equation*}
For any $q > 2$, let us choose $\varepsilon_0 := \frac{q C_1}{\pi} > 0$. For all $\varepsilon > \varepsilon_0$, it is guaranteed that $\frac{\pi}{\|f_\varepsilon(u)\|_{L^1(\Omega)}} > q$. Therefore,
\begin{equation*}
\int_\Omega e^{q|v(x)|} \, dx \le \int_\Omega \exp\left( \frac{\pi}{\|f_\varepsilon(u)\|_{L^1(\Omega)}} |v(x)| \right) \, dx \le C_2 e^{\pi K}.
\end{equation*}
Setting $C_q := C_2 e^{\pi K}$, which is clearly independent of $\varepsilon$ and $u$, completes the proof.
\end{proof}

Using Lemma \ref{lem:jensen} and standard elliptic regularity theory in \cite{GT}, we can improve the integrability to obtain a uniform $L^\infty$-bound.

\begin{proposition}\label{prop:L_infty}
There exists a constant $M > 0$, independent of $\varepsilon$, such that for all $\varepsilon$ sufficiently large, any solution $u$ to \eqref{eq:main} satisfies
\begin{equation*}
\|u\|_{L^\infty(\Omega)} \le M.
\end{equation*}
\end{proposition}
\begin{proof}
First, we establish a uniform upper bound for the spatial average $\bar{u} = \frac{1}{|\Omega|} \int_\Omega u \, dx$. 
By integrating the equation \eqref{eq:main} and using the Neumann boundary condition, we have
\begin{equation*}
\int_\Omega (e^u - 1 - au) \, dx = 0,
\end{equation*}
which implies
\begin{equation*}
\frac{1}{|\Omega|} \int_\Omega e^u \, dx = 1 + a \bar{u}.
\end{equation*}
Since $x \mapsto e^x$ is a convex function, Jensen's inequality yields
\begin{equation*}
e^{\bar{u}} = \exp\left( \frac{1}{|\Omega|} \int_\Omega u \, dx \right) \le \frac{1}{|\Omega|} \int_\Omega e^u \, dx = 1 + a \bar{u}.
\end{equation*}
Hence, $0 \le \bar{u} \le \xi_a$, where $\xi_a > 0$ is the unique positive root of $f(t) = 0$. Thus, $\bar{u}$ is bounded by $\xi_a$ uniformly for all $\varepsilon$.

Next, we bound the nonlinear term in $L^q(\Omega)$. Fix any $q > 2$. Then Lemma \ref{lem:jensen} guarantees that $\int_\Omega e^{q|v|} \, dx \le C_q$ for all $\varepsilon > \varepsilon_0$. Using $u = \bar{u} + v$ and $\bar{u} \le \xi_a$, we get
\begin{equation*}
\int_\Omega e^{q|u|} \, dx \le e^{q\bar{u}} \int_\Omega e^{q|v|} \, dx \le e^{q\xi_a} C_q.
\end{equation*}
Then we estimate the right-hand side as
\begin{align*}
\|f_\varepsilon(u)\|_{L^q(\Omega)} &\le \frac{1}{\varepsilon} \left(\|e^{|u|}\|_{L^q(\Omega)} + a\|u\|_{L^q(\Omega)}+|\Omega|^{1/q} \right)\\ 
&\le \frac{C}{\varepsilon}\left(\|e^{|u|}\|_{L^q(\Omega)}+1\right)\\
&\le \frac{C_q'}{\varepsilon} \le C_q'
\end{align*}
for all $\varepsilon \ge \max\{\varepsilon_0, 1\}$, where $C_q'$ is a constant independent of $\varepsilon$ and $u$.
Note that the inequality $|u| \le e^{|u|}$ was used to suppress $\|u\|_{L^q(\Omega)}$ by $\|e^{|u|}\|_{L^q(\Omega)}.$

Since $v$ solves the Neumann problem
\begin{equation*}
\begin{cases}
-\Delta v = f_\varepsilon(u) & \text{in } \Omega, \\
\frac{\partial v}{\partial \nu} = 0 & \text{on } \partial\Omega,
\end{cases}
\end{equation*}
with $\int_\Omega v \, dx = 0$, standard elliptic estimates provide a uniform bound in $W^{2,q}(\Omega)$:
\begin{equation*}
\|v\|_{W^{2,q}(\Omega)} \le C \|f_\varepsilon(u)\|_{L^q(\Omega)} \le C C_q',
\end{equation*}
where $C$ depends only on $q$ and $\Omega$. 
Because $q > 2$ and the spatial dimension is $2$, the Sobolev embedding theorem $W^{2,q}(\Omega) \hookrightarrow L^\infty(\Omega)$ ensures that
\begin{equation*}
\|v\|_{L^\infty(\Omega)} \le \tilde{C} \|v\|_{W^{2,q}(\Omega)} \le C_q''.
\end{equation*}

Finally, since $u = \bar{u} + v$, we conclude that
\begin{equation*}
\|u\|_{L^\infty(\Omega)} \le \bar{u} + \|v\|_{L^\infty(\Omega)} \le \xi_a + C_q'' := M,
\end{equation*}
where $M>0$ depends only on $a$ and $\Omega$.
\end{proof}

\section{Proof of the Main Theorem}
With the uniform $L^\infty$-estimate at hand, we now prove Theorem \ref{thm:main}. The proof relies on analyzing the Dirichlet energy of the fluctuation $v = u - \bar{u}$ and comparing it with the spectral gap of the Neumann Laplacian.

\begin{proof}
Since $\int_{\Omega}{v\ dx}=0$, multiplying equation \eqref{eq:main} by $v = u - \bar{u}$ and integrating by parts gives
\begin{equation} \label{eq:energy}
\varepsilon \int_\Omega |\nabla v|^2 \, dx = \int_\Omega (f(u) - f(\bar{u})) v \, dx.
\end{equation}
By the mean value theorem applied pointwise, we may write
\begin{equation*}
f(u) - f(\bar{u}) = f'(c(x))\, v(x)
\end{equation*}
for some $c(x)$ between $u(x)$ and $\bar{u}$. Since $\|u\|_{L^\infty} \le M$ by Proposition \ref{prop:L_infty}, $|f'(c(x))|$ is uniformly bounded by some constant $K:=\sup_{t \in [-M,M]}|f'(t)| = \max\{e^M-a, a-e^{-M}\}$. 
Applying the Poincaré inequality for functions with zero average,
\begin{equation*}
\int_\Omega |\nabla v|^2 \, dx \ge \mu_1 \int_\Omega v^2 \, dx,
\end{equation*}
where $\mu_1 > 0$ is the first nonzero Neumann eigenvalue, we deduce from \eqref{eq:energy} that
\begin{equation*}
\varepsilon \mu_1 \int_\Omega v^2 \, dx \le \int_\Omega f'(c(x))\, v^2 \, dx \le K \int_\Omega v^2 \, dx.
\end{equation*}
Hence,
\begin{equation*}
(\varepsilon \mu_1 - K)\int_\Omega v^2\,dx \le 0.
\end{equation*}
In particular, if $\varepsilon > K/\mu_1$, then the coefficient $\varepsilon \mu_1 - K$ is strictly positive, which forces $\int_\Omega v^2\,dx = 0$, and therefore $v \equiv 0$.
Consequently, $u \equiv \bar{u}$ is constant.

\end{proof}

\bibliographystyle{amsplain}

\end{document}